\documentclass[12pt]{amsart}

\usepackage{amstext}
\usepackage{amsthm}
\usepackage{amsmath}
\usepackage{amssymb}
\usepackage{latexsym}
\usepackage{amsfonts}
\usepackage{color}
\usepackage{graphicx}
\usepackage{yfonts}

\makeatletter
\renewcommand*\env@matrix[1][*\c@MaxMatrixCols c]{%
  \hskip -\arraycolsep
  \let\@ifnextchar\new@ifnextchar
  \array{#1}}
\makeatother

\AtBeginDocument{%
 \def\MR#1{} 
}

\usepackage[pagebackref,hypertexnames=false, colorlinks, citecolor=black, linkcolor=black, urlcolor=red]{hyperref}
\usepackage[backrefs]{amsrefs}

\usepackage{latexsym}
\usepackage{amssymb}
\usepackage{euscript}

\def\mcc{M\raise.5ex\hbox{c}C}
\def\mccarthy{M\raise.5ex\hbox{c}Carthy}

\def\vare{\varepsilon}

\let\i=\infty

\def\={\ = \ }

\def\C{\mathbb C}

\def\T{\mathbb T}
\def\D{\mathbb D}
\def\B{\mathbb B}

\def\be{\setcounter{equation}{\value{theorem}} \begin{equation}}
\def\ee{\end{equation} \addtocounter{theorem}{1}}
\def\beq{\begin{eqnarray*}}
\def\eeq{\end{eqnarray*}}

\def\vs{\vskip 5pt}

\def\bp{{\sc Proof: }}
\def\ep{{}{\hfill $\Box$} \vskip 5pt \par}

\def\bl{\begin{lemma}}
\def\el{\end{lemma}}
\def\bt{\begin{theorem}}
\def\et{\end{theorem}}
\def\bprop{\begin{prop}}
\def\eprop{\end{prop}}
\def\bd{\begin{definition}}
\def\ed{\end{definition}}
\def\br{\begin{remark}}
\def\er{\end{remark}}
\def\bexer{\begin{exercise}}
\def\eexer{\end{exercise}}

\newtheorem{theorem}{Theorem}[section]
\newtheorem{prop}[theorem]{Proposition}
\newtheorem{lemma}[theorem]{Lemma}

\newtheorem{definition}[theorem]{Definition}

\renewcommand\Re{\mathrm{Re\, }}
\renewcommand\Im{\mathrm{Im\, }}

\def\vare{\varepsilon}

\newcommand\vp{\varphi}

\numberwithin{equation}{section}
\title{The Julia-Carath\'eodory theorem on the bidisk revisited}
\author{John E. M\raise.5ex\hbox{\tiny c}Carthy}
\thanks{Partially supported by National Science Foundation Grant  
DMS 1565243}
\author{
James E. Pascoe}
\thanks{Partially supported by National Science Foundation Fellowship  
DMS 1606260}

\begin{document}

\date{\today}

\begin{abstract}
The Julia quotient measures the ratio of the distance of a function value from the boundary to the distance from the boundary.
The Julia-Carath\'eodory theorem on the bidisk states that if the Julia quotient is bounded along some sequence of nontangential approach to some point in the torus, the function must have directional derivatives in all directions pointing into the bidisk. The directional derivative, however, need not be a linear function of the direction in that case. 
In this note, we show that if the Julia quotient is uniformly bounded along every sequence of nontangential approach, the function must have a linear directional derivative. Additionally, we analyze a weaker condition, corresponding to being Lipschitz near the boundary, which implies the existence of a linear directional derivative for rational functions.
\end{abstract}

\bibliographystyle{plain}
\maketitle

\section{Introduction}
\label{seca}



Consider the following four holomorphic functions. They all map the bidisk $\D^2$ to $\overline{\D}$,
and they all have a singularity at $(1,1)$. However, the nature of the singularity gets progressively worse.

\beq
\phi_1(z) &\= & \frac{-4 z_1 z_2^2 + z_2^2 + 3 z_1 z_2 - z_1 + z_2}{
z_2^2 - z_1 z_2 - z_1 - 3 z_2 +4}
\\
\phi_2(z) &\= & 
 \frac{(z_2 -z_1) -2(1-z_1)(1-z_2)  \log (\frac{1+z_2}{1-z_2} \frac{1-z_1}{1+z_1} ) } 
{(z_2 -z_1)  + 2(1-z_1)(1-z_2)  \log (\frac{1+z_2}{1-z_2} \frac{1-z_1}{1+z_1})  } 
\\
\phi_3(z) &\= & 
\frac{3z_1 z_2 - 2z_1 - z_2}{3 - z_1 - 2 z_2} \\
\phi_4(z) &\= & \prod_{n=1}^\i  
\frac{2(1-2^{-n}) - z_1 - z_2}{2 -(1 - 2^{-n})(z_1 + z_2)}
 .
\eeq

In one variable, the Julia-Carath\'eodory theorem asserts that for a holomorphic function 
that maps $\D$ to $\D$, a relatively mild regularity assumption at a boundary point results in 
seemingly stronger forms of regularity \cite{ju20, car29}. Variants of their results on the bidisk have been studied in \cite{wlo87, jaf93, ab98,amy12a, amy11b,aty12,bk13}.
In \cite{amy12a} two notions of regularity, a weaker one called a B point and a stronger one called a C point, were studied (we shall give definitions in Section~\ref{secb}).
Letting $\chi$ denote the point $(1,1)$ on the distinguished boundary of $\D^2$,
the function $\phi_1$ above is the only one with  a C point at $\chi$;  the first three functions have a B point 
at $\chi$, and $\phi_4$ does not.

In this note, we  give an alternative characterization of C points (Theorem~\ref{mainresult}).
We also introduce a new notion of regularity, called a B+ point, which is intermediate between
B point and C point.  We analyze these points in Theorem \ref{bplusmain}), and show that
the function $\phi_2$ above has a B+ point at $\chi$, and $\phi_3$ does not.
We prove that if a function is rational, every B+ point is actually a C point (Theorem \ref{thmrat}).

\section{B, B+ and C points}
\label{secb}

In the bidisk, we shall measure distances with the supremum norm.
If $\tau \in \partial \D^2$, we shall say that a set $S \subset \D^2$ is 
non-tangential at $\tau$ if there exists a constant $M > 0$ such that
\be
\label{eqb0}
{\rm dist } (z, \tau) \ \leq \ M \ {\rm dist}(z,\partial \D^2) \quad \forall \ z \in S.
\ee
The least $M$ that satisfies \eqref{eqb0} we shall call the {\em aperture} of $S$.
We shall say that a statement holds non-tangentially at $\tau$ if, for every set
$S$ that is non-tangential at $\tau$, there exists $\vare > 0$ such that the
statement holds on $S \cap \B(\tau, \vare)$.

We shall let $\Pi$ denote the upper half plane. A Pick function in $d$ variables
is a holomorphic function from $\Pi^d$ to $\overline{\Pi}$.

For the rest of this section, we shall assume that 
$\varphi: \mathbb{D}^2 \rightarrow \overline{\mathbb{D}}$ is analytic
	and $\tau \in \T^2$ is in the distinguished boundary of the bidisk.

	The point $\tau$ is called a B point for $\vp$ if either (A) or (B) of the following equivalent conditions hold:
	
\bt
\label{thmb1} \cite{amy12a} The following are equivalent:
\[
\text{(A)} \qquad \liminf_{z \rightarrow \tau}  \frac{1 - |\varphi(z)|}{1-\|z\|} < \infty .
\]
(B) 	There exists an $\omega \in \mathbb{T}$
		such that 
		for every nontangential set $S$ there exists 
	 $\alpha >0,   \varepsilon >0,$ such that
		for every $z \in S \cap B(\tau,   \varepsilon),$
	\be
	\label{eqb1}
	\|\varphi(z) - \omega\| \leq \alpha\|\tau-z\|.\ee
Moreover, if these conditions hold, then for every $h \in \C^2$
such that $\tau + th \in \D^2$ for small positive $t$, then the directional
derivative
\[
D\vp(\tau) [h] \ := \
\lim_{t \downarrow 0} \frac{\phi(\tau + th) - \omega}{t} \]
exists, and is given by
\be
\label{eqb2}
D\varphi(\tau)[h] \=  - \omega\overline{\tau_2}h_2 \ \eta\left(
		\frac{\overline{\tau_2}h_2}{\overline{\tau_1}h_1}\right) , \ee
		where $\eta(w)$ and $-w\eta(w)$ are Pick functions in one variable.
			\et

We shall say that a point is a B+ point if the constant $\alpha$ in \eqref{eqb1}
can be chosen independently of the non-tangential set $S$.

\begin{definition}
We say that $\varphi$ has a \emph{$B^+$-point} at $\tau$ if 
		there exists an $\omega \in \mathbb{T},$ $\alpha >0,$ such that
		for every nontangential set $S$
		there exists an $  \varepsilon >0,$ such that
		for every $z \in S \cap B(\tau,   \varepsilon),$
			$$\|\varphi(z) - \omega\| \leq \alpha\|\tau-z\|.$$
			\end{definition}
		
			We show the following theorem characterizing $B^+$ points.
\begin{theorem}\label{bplusmain}
	The following are equivalent:
	\begin{enumerate}
		\item $\varphi$ has a $B^+$ point at $\tau.$
		\item $\varphi$ has a $B$-point at $\tau$ and for some constant $\alpha$
			$$|D\varphi(\tau)[h]| \leq \alpha \|h\|.$$
		\item There exists an $\omega \in \mathbb{T}$, non-negative real numbers
		$A $ and $B$, 
		and a bounded analytic function $g$ on the domain $\C \setminus [-1,1]$  such that
			$$D\varphi(\tau)[h] = Ah_1 + Bh_2  +
		(h_2-h_1)g\left(\frac{h_1+h_2}{h_1-h_2}\right).$$
	\end{enumerate}
\end{theorem}	
	
The point $\tau$ is called a C point for $\vp$ if either condition (C) or (D) below holds.
They say that $D\phi(\tau)[h]$ is a linear function of $h$.
\bt
\cite{amy12a}
The following are equivalent:

(C)		There exist $\omega \in \mathbb{T}$ and $  \lambda \in \mathbb{C}^2$
such that
		for every nontangential set $S$ and every $\beta >0$,
		there exists 		
		$  \varepsilon >0$ such that
		for every $z \in S \cap B(\tau,   \varepsilon),$
			$$\|\varphi(z) - \omega -  \lambda \cdot z\| \leq \beta\|\tau-z\|.$$
	
	(D) 	There exist $\omega \in \mathbb{T}$ and $  \lambda \in \mathbb{C}^2$
such that	$\lim_{z \xrightarrow{nt} \tau} \vp(z) = \omega$ and
$\lim_{z \xrightarrow{nt} \tau} \nabla \phi(z) =  \lambda$.
			\et


Thus:
\begin{itemize}
	\item At a $B$-point,
		$$\varphi(z) = \omega + O(\tau-z) \text{ nontangentially}.$$
	\item At a $B^+$-point,
		$$\varphi(z) = \omega + O(\tau-z) \text{ uniformly nontangentially}.$$
	\item At a $C$-point
		$$\varphi(z) = \omega + \eta \cdot z + o(\|\tau-z\|) \text{ nontangentially}.$$
\end{itemize}

We prove that if the $\liminf$ in condition (A) is replaced by a non-tangential $\limsup$,
we get another characterization of C points:
\begin{theorem}\label{mainresult}
	Let $\varphi: \mathbb{D}^2 \rightarrow \overline{\mathbb{D}}$ be analytic
	and $\tau \in \mathbb{T}^2.$
	If
		$$\limsup_{z \xrightarrow{nt} \tau}  \frac{1 - |\varphi(z)|}{1-\|z\|} < \infty,$$
	then $\varphi$ has a $C$-point at $\tau.$
	(Here the $\limsup$ is over all sequences which approach $\tau$ nontangentially.)
\end{theorem}

We also show that for rational functions, B+ points are C points.
\bt
\label{thmrat}
Let $\vp : \D^2 \to \overline{\D}$ be rational. Then if $\tau \in \partial \D^2$
is a B+ point for $\vp$, it is a C point.
\et

\section{Proofs}
\label{secc}

In \cite[Thm. 5.3]{aty12}, Agler, Tully-Doyle and Young gave the following characterization
of the functions $\eta$ arising in Theorem~\ref{thmb1}. In \cite[Thm 6.2]{aty12},
they showed that all such $\eta$ can arise. (We introduce a factor of 4 in the following theorem just
for notational convenience later).

\bt
\label{thmc1}
\cite{aty12}
The functions $\eta(w)$ and $-w\eta (w)$ are both in the Pick class of one variable
if and only if there is a positive measure $\mu$ supported on $[-1,1]$ such that
\be
\label{eqc1}
\eta(w) \= - 4
\int_{-1}^1 \frac{1}{(1-t) + (1+t) w}
d\mu(t).
\ee
Moreover, for any $\eta$ satisfying \eqref{eqc1}, there exists a holomorphic
$\vp : \D^2 \to \D$ with a B point at $\chi$ satisfying \eqref{eqb2}.
\et
Note that any function $\eta$ given by \eqref{eqc1} is analytic on $\C \setminus (-\infty,0]$
and is real-valued on the positive real axis.

A Pick function of two variables is called homogeneous of degree one if
	$$f(wz_1,wz_2) \= wf(z_1,z_2)$$
	whenever $(z_1,z_2), (wz_1, wz_2) \in \Pi^2$.
	There is a one-to-one correspondence between homogeneous Pick functions
	of two variables and functions $\eta$ of the form \eqref{eqc1}.
	\bprop
	\label{prc1}
	Let $f$ be a homogeneous Pick function of degree one in 2 variables.
Define $\eta$ by
\be
\label{eqc2}
\eta(w) \= e^{i s} f(- \frac{e^{-is}}{w}, - e^{-is}) 
\ee
whenever $(- \frac{e^{-is}}{w}, - e^{-is}) $ is in $\Pi^2$.
Then $\eta$ is well-defined, and both $\eta(w)$ and $-w \eta(w)$ are in the Pick
class of one variable. Conversely, if  $\eta(w)$ and $-w \eta(w)$ are in the Pick
class, then the function $f$ defined by
\be
\label{eqc3}
f(z_1, z_2) \=
\left\{
\begin{array}{l}
-z_2
\eta( \frac{z_2}{z_1} )  \quad 
\text{if }\  
\frac{z_2}{z_1} 
\in \Pi \\
\\
- z_2 \overline{\eta( \frac{\bar z_2}{\bar z_1} )} \quad \text{if\ } 
 \frac{\bar z_2}{\bar z_1} \in \Pi 
\end{array} \right.
\ee
is a homogeneous Pick function of degree one.
\eprop
\bp
Let $f$ be given. Homogeneity means \eqref{eqc2} is independent of $s$. For any $w \in \Pi$, if $s$ is positive and small enough, the point $(- \frac{e^{-is}}{w}, - e^{-is}) $ is in $\Pi^2$, and letting $s \downarrow 0$, we get that
$\eta(w) \in \Pi$. Similarly, for $s$ positive and sufficiently small,
\[
-w \eta (w) \= e^{is/2} f(e^{-is/2}, e^{-is/2} w)
\]
is in $\Pi$.

Conversely, let $\eta$ be given.
Since $\eta$ is real-valued on the positive real-axis, both definitions of $f$ agree when $z_2$ is a positive multiple of $z_1$,
so by continuity \eqref{eqc3} defines $f$ for any $z \in \Pi^2$.
By inspection, the formula is homogeneous of degree one. It remains to show that $f$ maps $\Pi^2$ to $\Pi$.

Let $z \in \Pi^2$, and first assume that $\arg(z_2/z_1) \in [0,\pi)$.
Let $\theta_r = \arg(z_r), \ r = 1,2$. Then as both $\eta(z_2/z_1)$ and
$-(z_2/z_1) \eta(z_2/z_1) $ are in $\Pi$, we have
\beq
0 &\ < \ \arg \eta(z_2/z_1) & \ < \ \pi \\
0 &\ < \ \arg \eta(z_2/z_1) + \theta_2 - \theta_1 - \pi & \ < \ \pi.
\eeq
By \eqref{eqc3}, we have
\[
\arg f(z_1, z_2) \= \theta_2 + \arg \eta(z_2/z_1) - \pi,
\]
and this is sandwiched between $\theta_1$ and $\theta_2$.
Similarly, if $\arg(z_2/z_1) \in (-\pi,0]$, we get that 
$\arg f(z_1, z_2)$ is between $\theta_2$ and $\theta_1$.
\ep

Combining Proposition~\ref{prc1} with Theorem \ref{thmc1}, we get that a function $f$
is in the Pick class of two variables and homogeneous of degree one if and only if it
can be represented as
\be
\label{eqc4}
f(z_1,z_2) \=  4\ \int_{-1}^1 \frac{z_1z_2}{(1+t)z_1+(1-t)z_2}d\mu(t).
\ee
Moreover, if $\vp$  has a B point at $\tau$, then \eqref{eqb2} becomes
\be
\label{eqc45}
D\varphi(\tau)(i\tau_1 z_1, i\tau_2 z_2)
\=
i \vp(\tau) f(z_1,z_2) .
\ee

We shall rewrite \eqref{eqc4} in a way that makes it easier to read off regularity.

\begin{lemma}
\label{resolvedrepresentation}
	If $f$ is of the form \eqref{eqc4}, then $f$ can be written
	\be
	\label{eqc5}
	f(z_1,z_2) \=  (\mu_0 - \mu_1)z_1 + (\mu_0 +\mu_1)z_2  
	+ (z_2-z_1)g\left(\frac{z_1+z_2}{z_2-z_1}\right),
	\ee
	where $\mu_i$ is the $i$-th moment of $\mu$
	and $g$ is a Pick function of one variable such that $\Im g$ is the Poisson integral of the measure $(1-t^2) d\mu(t)$.
\end{lemma}
\begin{proof}
Let us  calculate $f(\zeta-1,\zeta+1)$ for $\zeta \in \Pi$.
	\begin{align*}
	f(\zeta-1,\zeta+1) & = 4\int_{-1}^1 \frac{(\zeta-1)(\zeta+1)}{(1+t)(\zeta-1)+ (1-t)(\zeta+1)}d\mu(t)\\
		&= 2\int_{-1}^1\frac{1 - \zeta^2 }{t-\zeta}d\mu(t) \\
	&=2 \int_{-1}^1 \frac{1-t^2}{t-\zeta} +  \frac{t^2-\zeta^2}{t-\zeta}d\mu(t) \\
	&=2 \int_{-1}^1 \frac{1-t^2}{t-\zeta} + \zeta  + td\mu(t) \\
	&= 2\mu_0\zeta + 2\mu_1 + 2\int_{-1}^1 \frac{1-t^2}{t-\zeta}d\mu(t)\\
	&= 2\mu_0\zeta + 2\mu_1 + 2\int_{-1}^1 \frac{1}{t-\zeta}(1-t^2)d\mu(t)\\
	\end{align*}
	Define the Pick function $g$ by
	\be
	\label{eqc6}
	g(\zeta) = \int_{-1}^1 \frac{1}{t-\zeta}(1-t^2)d\mu(t). \ee
	 Notice that $g$ is holomorphic on $\C \setminus [-1,1]$, and satisfies
$g(\zeta) = \overline{g(\bar \zeta)}$.
Let $z \in \Pi^2$.
If $z_1 \neq z_2$, define
\beq
w &\= & \frac{z_2 - z_1}{2}\\
\zeta &\=& \frac{z_2 + z_1}{z_2 - z_1},
\eeq
so $z_1 = w(\zeta-1)$ and $z_2 = w(\zeta+1)$.
Since $z_2$ cannot be a negative multiple of $z_1$, we get that $\zeta$ is never
in $[-1,1]$, so 
 \eqref{eqc5} holds provided $z_1 \neq z_2$. By continuity, it holds for all $z \in \Pi^2$.
\end{proof}
Note that from \eqref{eqc6}, we see that $\limsup_{s \to \infty} |is g(is)| < \infty$.
Using Lemma \ref{resolvedrepresentation}, we obtain that
a rational homogeneous Pick function which is linearly bounded must be linear, which in turn
will imply that for rational functions, a $B^+$ point is a $C$-point.
\begin{lemma}\label{rationalresolved}
	Let $f$ be a  homogeneous Pick function of degree $1$ in two variables represented as in \eqref{eqc5}.
	There exists a $C>0$ such that
		$$\|f(z_1,z_2)\| \leq C\|z\|$$
	if and only if $g$ is a bounded function on $\C \setminus [-1,1]$.
	If in addition $f$ is rational, then $g$ is identically $0$.
\end{lemma}
\begin{proof}
The	$(\Leftarrow)$ implication is trivial. For the converse, assume $f$ is 
 a  homogeneous Pick function of degree $1$ in two variables,
		and represent it as in \eqref{eqc5}. Then we have
\beq
C\|z\| & \ \geq \ & \|f(z_1,z_2)\| \\
&\geq  & -\|(\mu_0 - \mu_1)z_1\| -\| (\mu_0 +\mu_1)z_2  \|
		+ \|(z_2-z_1)g\left(\frac{z_1+z_2}{z_2-z_1}\right)\| \\
		&\geq & -|\mu_0 - \mu_1|\|z\| - |\mu_0 +\mu_1|\|z  \|
		+ \|(z_2-z_1)g\left(\frac{z_1+z_2}{z_2-z_1}\right)\|.
		\eeq
		So
		$$ (C +2 \mu_0)\|z\|\geq 
		 \|(z_2-z_1)g\left(\frac{z_1+z_2}{z_2-z_1}\right)\|.$$
		Making the substitution $z = (\zeta - 1, \zeta +1),$
		we get that 
\be
\label{eqc7}
(C +2\mu_0)(\|\zeta\| +1) \geq \|2g\left(\zeta \right)\|
\ee
for all $\zeta \in \C \setminus [-1,1]$.
Since $g$ is the Cauchy transform of a compactly supported measure, it vanishes at $\i$,
and so \eqref{eqc7} yields that $g$ is bounded on $\C \setminus [-1,1]$.

If $f$ is rational, then 
\[
g(\zeta) \= \frac 12 [ f(\zeta - 1,\zeta +1) ] - \mu_0 \zeta - \mu_1 ,
\]
is also rational,  and since it vanishes at $\infty$ and has no poles, it is 
 is identically zero.
\end{proof}


{\sc Proof of Theorem~\ref{bplusmain}:}
$(1) \Leftrightarrow (2)$: By Theorems~\ref{thmb1} and \ref{thmc1}, the derivative of $\vp$ at $\tau$
is a  Pick function of two variables, homogeneous of degree $1$, which can therefore be represented as
in \eqref{eqc5}.
From \eqref{eqc45}, we have that for every $h$ that points into $\D^2$,  for $t > 0$:
\[
\vp(\tau + th) - \vp(\tau) \= i \vp(\tau)  tf(-i \bar \tau h_1, -i \bar \tau h_2) + o(t) .
\]
At a B+ point, the left-hand side is bounded by $\alpha t \| h\|$, which 
means 
\be
\label{eqc8}
| D \vp (\tau) [h] | \= |f(-i \bar \tau h_1, -i \bar \tau h_2)|
\ee
is bounded by $\alpha \| h \|$.

Conversely, assume
$$ | D \vp (\tau) [h] | \ \leq \  \alpha \| h \|,
$$
and fix a set $S$ that is non-tangential at $\tau$. Since $S$ is non-tangential, it is enclosed in a cone with apex at $\tau$
and some aperture $M$, 
so there exists some $\vare > 0$ and some compact connected set $K$ of vectors in $\C^2$ with unit
norm  such that
\[
S \cap B(\tau , \vare)  \ \subseteq\  \{ \tau + th: \ 0 < t < \vare, \ h \in K \} \ \subseteq \ \D^2 .
\]
For each $ 0 < t < \vare$, let
\[
\Psi_t (h) \= \frac{\vp(\tau + th) - \vp (\tau)}{t} .
\]
The functions $\Psi_t(h)$ are 
continuous on $K$, and 
\[
\lim_{t \downarrow 0} \Psi_t(h)  \= D\vp (\tau) [h]
\]
exists. Since $K$ is compact, the convergence is uniform. So for all $ \beta > 0$,
 there exists $\delta > 0$ such that
 \[
| \Psi_t(h) - D\vp(\tau)[h] | \ < \ \beta \quad \forall\ 0 < t < \delta .
\]
Therefore
\[
\| \vp(\tau + th) - \vp (\tau) | \ \leq \ (\alpha + \beta) t \quad \forall\ 0 < t < \delta ,
\]
which means
\[
\| \vp(z) - \vp (\tau) | \ \leq \ (\alpha + \beta) \| \tau - z \| \quad \forall\ z \in S \cap B(\tau , \delta).
\]

$(2) \Leftrightarrow (3)$
This follows from  \eqref{eqc8} and Lemma~\ref{rationalresolved}.
\ep

{\sc Proof of Theorem~\ref{thmrat}:}
If $\vp$ is rational, then $D \vp (\tau) [h]$ is a rational function of $h$, 
so $f$ is rational.
By Theorem~\ref{bplusmain} and Lemma~\ref{rationalresolved},
if $f$ is rational then $g$ is zero, so $f$ is linear, and hence $\tau$ is actually a C point.
\ep

We now prove two more lemmata for Theorem \ref{mainresult}.

\begin{lemma} \label{innerlinear}
	Let $f$ be a homogeneous Pick function of degree $1$ in two variables which satisfies
\be
\label{eqc10}
\Im f(z_1,z_2) \leq \gamma  \max(\Im z_1, \Im z_2)
\ee
	for some $\gamma \geq 0.$
	Then $f$ is linear.
\end{lemma}
\begin{proof}
	Let 
	$$f(z_1,z_2) =  (\mu_0 - \mu_1)z_1 + (\mu_0 +\mu_1)z_2  +
		(z_2-z_1)g\left(\frac{z_1+z_2}{z_2-z_1}\right)$$
as in \eqref{eqc5},
	where $g$ is the Cauchy transform of the measure  $(1-t^2) d\mu(t)$ on $[-1,1]$, and $\Im g$ is the
	Poisson transform of this measure.
	
	We will show $g$ must equal to $0.$
	
	Since
		$(\mu_0 - \mu_1)z_1 + (\mu_0 +\mu_1)z_2$   satisfies \eqref{eqc10}
		for some $\gamma$, 
	we get that 
		$$h(z_1,z_2) = (z_2-z_1)g\left(\frac{z_1+z_2}{z_2-z_1}\right)$$
	must satisfy, for some perhaps different $\gamma$, 
		$$\Im h(z_1,z_2) \leq \gamma \max(\Im z_1, \Im z_2).$$
	Under the substitution $z_1 = \zeta  - 1, z_2 = \zeta + 1$ we get that
\be
\label{eqc11}
2 \Im g(\zeta) \leq \gamma\Im \zeta.
\ee
As $\Im g$ is bounded in $\Pi$, and $\Im g(t+iy) dt$ converges weak-star to
$(1-t^2) d\mu(t)$ as $y \to 0^+$, \eqref{eqc11} shows that $\Im g $, and hence $g$,  is identically $0$.
\end{proof}

\begin{lemma} \label{ujqinner}
	Suppose 
		$$\gamma =\limsup_{z \xrightarrow{nt} \tau}  \frac{1 - |\varphi(z)|}{1-\|z\|}$$
	is finite. The Pick function in two variables
	 $$f(z_1,z_2) = -i\overline{\varphi(\tau)}D\varphi(\tau)(i\tau_1 z_1, i\tau_2 z_2)$$ satisfies
\be
\label{eqc9}
\Im f(z_1,z_2) \leq \gamma \max (\Im z_1, \Im z_2).
\ee
\end{lemma}
\begin{proof}
Without loss of generality, assume $\tau = \chi = (1,1)$ and $f(\tau)=1.$
	Note
	$\limsup_{z \xrightarrow{nt} \tau}  \frac{1 - |\varphi(z)|^2}{1-|z_r|^2}$
	is finite for $r=1,2$.
	Consider the limit along the ray $z = 1 - th$ as $t \downarrow 0,$
for any $h$ pointing into $\D^2$. Since we chose $\tau = \chi$, this means
$\Re(h_r) < 0$ for $r=1,2$.
	$$
\limsup_{z \xrightarrow{nt} \tau}  \frac{1 - |\varphi(1-th)|^2}{1-|1-th_r|^2} \ \leq \ \gamma.$$
	So 
		$$\limsup_{t \rightarrow 0^+} 
		\frac{1 - (\Re \varphi(1-th))^2-(\Im \varphi(1-th))^2}{1-(\Re1-th_r)^2-(\Im1-th_r)^2} \ \leq \ \gamma.$$
	Applying l'Hospital's rule, we get
		$$\frac{\Re D\varphi(\tau)[h]}{\Re h_r} \leq \gamma.$$
	That is,
		$$\Re D\varphi(\tau)[h] \leq \gamma\Re h_r.$$
As $z = ih$, we get \eqref{eqc9}.
\end{proof}
{\sc Proof of Theorem~\ref{mainresult}:} This now follows from Lemmata \ref{ujqinner} and \ref{innerlinear}.
\ep

\section{Examples}
\label{secd}

Consider the functions $\phi_1$ to $\phi_4$ from the Introduction.
Note that 
\[
\phi_4(r,r) \= B(r)
\]
where $B$ is the Blaschke product with zeroes at $(1- 2^{-n})$.
This is an interpolating sequence, so even  the limit
$\lim_{r \uparrow 1} |\phi_4(r,r)|$ does not exist.
In particular,  $\chi$ fails to be a B point for $\phi_4$.

\vs
For $\phi_3$, we get
\[
\lim_{r \uparrow 1} \frac{1 - |\phi(r \chi)|}{1-r} \= 1 ,
\]
so $\chi$ is a B point.
A calculation yields
\[
D \phi_3 (\chi) [h]
\=
- \frac{3h_1 h_2}{h_1 + 2 h_2} .
\]
Choosing $h = (- \vare +2 i , -\vare -i)$, we see that this is not $O(\|h\|)$, so $\chi$ is not a B+ point.

\vs
The function $\phi_1$ is given in \cite{amy12a} as an example of a function with a C point.
It is easy to check that
\[
\phi_1(r,r) \= r^2 ,
\]
so $\chi$ is a B point. By a straightforward but lengthier calculation,
\[
D\phi_1(\chi) [h] \= 2 h_2 ,
\]
which is linear, so $\chi$ is a C point.

\vs
The formula for $\phi_2$ reduces to $0/0$ when $z_1 = z_2$;
on the diagonal $\phi_2 (z)$ should be defined to equal $\frac{-3+5z_1}{5- 3z_1 } $.
We claim that this will give a function in the Schur class that has a B+ point at $\chi$.

To show this, we shall start with $g$ from \eqref{eqc6}, where we choose
$\mu$ to be linear Lebesgue measure on $[-1,1]$; so
\beq
g(\zeta) &\=& \int_{-1}^1 \frac{1-t^2}{t-\zeta}dt \\
&=& (1-\zeta^2) \log \left( \frac{\zeta -1}{\zeta+1} \right) - 2 \zeta .
\eeq
The function $g$ is bounded on $\C \setminus [-1,1]$.
From \eqref{eqc5},
\beq
f(z_1,z_2) &\= & 2 z_1 + 2 z_2 + (z_2 - z_1) g\left (\frac{z_2 + z_1}{z_2 - z_1} \right)\\
&=& 
\begin{cases} \frac{-4z_1z_2}{z_2 - z_1} \log (\frac{z_1}{z_2}) &\quad z_1 \neq z_2 \\
-4 z_1 &\quad z_1 = z_2
\end{cases}
 .
\eeq
From \eqref{eqc1}, we have
\[
\eta(w) \= - 4 \int_{-1}^1 \frac{dt}{(1-t) + (1+t) w} .
\]
Now, we follow the recipe from \cite{aty12} to produce a function with a B point at $\chi$ and slope 
function $\eta$.
For each $t \in [-1,1]$, let
\[
\psi_t(z) \= \left( (1-t) \frac{1+z_1}{1-z_1} + (1+t) 
\frac{1+z_2}{1-z_2} \right)^{-1} ,
\]
and let
\[
\psi(z) \= 4 \int_{-1}^1 \psi_t(z) dt .
\]
Each $\psi_t$, and hence also $\psi$, maps $\D^2$ to the right half-plane.
Moreover $\chi$ is a B point for $\psi$, and
\[
D \psi(\chi) [h] \= - 2 \int_{-1}^1 \frac{h_1 h_2}{(1-t) h_1 + (1+t) h_2} dt .
\]
Finally, we let 
\beq
\phi_2(z) &\= &  \frac{1-\psi(z)}{1 + \psi(z)} \\
&\=& 
\begin{cases} \frac{(z_2 -z_1) -2(1-z_1)(1-z_2)  \log (\frac{1+z_2}{1-z_2} \frac{1-z_1}{1+z_1} ) } 
{(z_2 -z_1)  + 2(1-z_1)(1-z_2)  \log (\frac{1+z_2}{1-z_2} \frac{1-z_1}{1+z_1})  } 
&\quad z_1 \neq z_2 \\
\frac{-3+5z_1}{5- 3z_1 }  &\quad z_1 = z_2
\end{cases}
 .
\eeq
Then $\lim_{r \uparrow 1} \phi_2(r,r) = 1$ and
\beq
D \phi_2 (\chi) [h] &\=&  -2 D \phi(\chi) [h] \\
&=& - h_2 \eta(h_2/h_1),
\eeq
so $\phi_2$ has the required slope function, and hence $\chi$  is a B+ point.




\end{document}